\documentclass[a4paper,11pt]{amsart}

\usepackage{latexsym}
\usepackage{amssymb,amsmath}
\usepackage{graphics}
\usepackage{url}
\usepackage[vcentermath]{youngtab}
  
\usepackage{booktabs,longtable, fancyvrb, color}  
\usepackage{geometry}
\usepackage{lscape}

\newtheorem{theorem}{Theorem}
\newtheorem*{theorem*}{Theorem}

\newtheorem{proposition}[theorem]{Proposition}

\newtheorem{lemma}[theorem]{Lemma}

\theoremstyle{definition}

\newcommand{\N}{\mathbf{N}}
\newcommand{\Z}{\mathbf{Z}}

\newcommand{\e}{\mathrm{e}}

\renewcommand{\epsilon}{\varepsilon}

\DeclareMathOperator{\Sym}{Sym}

\renewcommand{\theta}{\vartheta}

\newcommand{\mfrac}[2]{{\textstyle\frac{#1}{#2}}}

\newcounter{thmlistcnt}
	{\setcounter{thmlistcnt}{0}%
	\begin{list}{\emph{(\roman{thmlistcnt})}}{%
		\usecounter{thmlistcnt}%
		\setlength{\topsep}{0pt}%
		\setlength{\leftmargin}{0pt}%
		\setlength{\itemsep}{0pt}%
		\setlength{\labelwidth}{17pt}
		\setlength{\itemindent}{30pt}}%
	}%
	{\end{list}}%
	
\newcommand{\hf}[1]{\texttt{#1}}

\begin{document}
\newgeometry{margin=1.35in}
\title[]{Computing derangement probabilities of the symmetric group acting on $k$-sets}
\author{John R.~Britnell and Mark Wildon}
\date{\today}

\maketitle
\thispagestyle{empty}

\begin{abstract}
Let $i(\infty ,k)$ be the limiting proportion, as $n \rightarrow \infty$,
of permutations in the symmetric group of degree $n$ that fix a $k$-set. 
We give an algorithm for computing $i(\infty ,k)$
and state the values of $i(\infty ,k)$ for $k \le 30$. These values are consistent with a
conjecture of Peter Cameron that $i(\infty,k)$ is a decreasing function of $k$.
\end{abstract}

\section{Introduction}

The symmetric group $\Sym_n$ acts on the set of $k$-subsets of $\{1,\ldots, n\}$.
Let $i(n,k)$ be the proportion of permutations in $\Sym_n$ that fix at least
one such $k$-subset. Let $i(\infty, k) = \lim_{n \rightarrow \infty}
i(n,k)$. (We include below a short proof that this limit exists for all $k \in \N$.)
It was shown by Luczak and Pyber in \cite[\S3, Lemma]{LuczakPyber} that $i(n,k) < Ck^{-1/100}$ 
for some constant $C$, uniformly in $n$. Thus $\lim_{k \rightarrow \infty} i(\infty, k) = 0$.
Peter Cameron has conjectured \cite{CameronConj} that $i(\infty,k)$ is a decreasing
function of $k$. In this note we give an efficient algorithm for computing $i(\infty,k)$
and use it to prove that Cameron's conjecture is true for $k \le 30$.

We also compute the values of $i(n,k)$ for all $n \in \N$ such that $n \le 70$.
As a corollary, we find that if $2k \le n \le 70$ then $i(n,k) < i(n,k+1)$ if and only if 
\[ (n,k) \in \left\{ \begin{matrix}
(30,9),  (36,11), (39,12), (42, 13), (45, 14), (47,15),(48,15), \\
(51,16), (53,17), (54,17), (57,18),  (59,19),  (60,19), (63,20), \\
(64,21), (65,21), (66,21), (68,22), (69,22), (70,23) \end{matrix} \right\} \]
However it is consistent with our data
that $i(n,k) > i(n,k+1)$ for all $n$ and $k$ such that $k < n/4$, so these examples
do not rule out the approach to Cameron's conjecture through careful estimation of $i(n,k)$.
(Of course the choice of $n/4$ is slightly arbitrary: any function $f : \N \rightarrow \N$
such that $f(n) \rightarrow \infty$ as $n \rightarrow \infty$ and
 $i(n,k) > i(n,k+1)$ for all $n$ and $k$ with $k < f(n)$ would suffice.)
 
Another motivation for this note is recent work of Eberhard, Ford and Green \cite{EberhardFordGreen}.
The main theorem of \cite{EberhardFordGreen} states that there exist constants $A$ and $B$ such that 
\[ Ak^{-\delta}(1+ \log k)^{-3/2} \le  i(n,k) \le B k^{-\delta}(1+ \log k)^{-3/2} \]
for all $k$, $n \in \N$, where $\delta = 1 - \frac{1 + \log \log 2}{\log 2} \approx 0.0861$.
This paper cites earlier data collected by the present authors, using the algorithm
described below, that proves Cameron's conjecture for $k \le 23$.

\subsection*{Outline}
In \S 2 we recall the necessary background on cycle statistics in permutations. In \S 3 we describe
the `Derangement Table Algorithm' for computing $i(\infty,k)$. This algorithm was inspired by a method for
calculating $i(\infty ,k)$ by hand, shown to the authors by Peter M.~Neumann.
 In Appendix~A we discuss some features of the Haskell
implementation of this algorithm. Appendix B gives our data for $i(\infty,k)$ for $k \le 30$.
It is routine to compute $i(n,k)$ for small values of $n$ by exhausting over
all partitions of $n$. Appendices C and D give the values of $i(n,k)$ and $1-i(n,k)$ 
for $n \le 70$ and $k \le 35$.

\section{The limiting distribution of $k$-cycles in permutations}

Let $X^{(n)}_k(\pi)$ be the number of $k$-cycles in the permutation $\pi$, chosen
uniformly at random from $\Sym_n$.
Let $X_1, X_2, \ldots$ be independent Poisson random variables such that $X_k$ has mean $1/k$.
The following proposition is well known. It is proved as Theorem 1 in \cite{ArratiaTavare}.

\begin{proposition}\label{prop:limitdist}
Let $m \in \N$.
As $n \rightarrow \infty$, 
\[ (X^{(n)}_1, \ldots, X^{(n)}_m) \xrightarrow{\ \mathrm{dist}\ } (X_1, \ldots, X_m). \]
\end{proposition}

For $n \in \N$, 
let $P^{(n)}$ 
be the random partition 
having exactly $X^{(n)}_j$ parts of size $j$
for each $j \in \{1,\ldots,k\}$. 
It is clear that $i(\infty ,k)$ is the limit
as $n \rightarrow \infty$ of the probability that $P^{(n)}$ has a subpartition of size $k$.
It follows from Proposition~\ref{prop:limitdist} that $i(\infty,k)$ exists, and 
is equal to the probability that 
the random partition $ (1^{X_1},2^{X_2},\ldots,k^{X_k})$ with exactly $X_j$ parts of size $j$
has a subpartition of size $k$.

\section{The Derangements Table Algorithm}

The input of the Derangements Table Algorithm is a natural number $k$.
We call elements of $\N_0^k$ \emph{rows}, and elements of $\N_0^\ell$ 
for $\ell \le k$ \emph{partial rows}.
Say that a partial row $(m_1,\ldots,m_\ell)$ is \emph{$k$-free} if the partition $(1^{m_1}, \ldots, \ell^{m_\ell})$
has no subpartition of size~$k$. The output of the algorithm is a list in lexicographic order
from greatest to least of all $k$-free rows $(m_1,\ldots, m_k)$ with $m_j \le k/j$ for each~$j$.

\subsection*{Constructing $k$-free rows}

The algorithm's internal state is a partial row~$r$. At the start, $r$ is set  to $(k-1)$.

\begin{itemize}
\item[(A)] [Building a partial row] Suppose $r = (m_1,\ldots,m_\ell)$. 
\begin{itemize}
\item[(1)] If $\ell = k$ then output $r$ [$r$ is a row] and go to (B).

\item[(2)] Else, set $j = \ell+1$.
Take $m$ maximal such that $0 \le m < k/j$ and the partial row $(m_1,\ldots,m_\ell,m)$ is $k$-free.
Set $r$ equal to $(m_1,\ldots,m_\ell,m)$ and repeat (A).
\end{itemize}

\item[(B)] [Begin a new partial row] 
\begin{itemize}
\item[(1)] If $r = (0,\ldots,0) \in \Z^k$ then terminate. 

\item[(2)] Else, we have $r = (m_1,\ldots,m_j,0,\ldots, 0) \in \Z^k$
where $m_j \ge 1$. Set $r$  to $(m_1,\ldots,m_j-1)$ and go to (A).
\end{itemize}
\end{itemize}

The details of the implementation of Step (A2) are key to the speed of the algorithm. We describe the
feature of the greatest mathematical interest here and leave the other refinements to Appendix A.

Say that a partition $\lambda$ of $n$ is \emph{$t$-universal} if $\lambda$ has subpartitions
of all numbers $s \le t$. There is a surprisingly simple characterization of universal partitions.

\begin{proposition}\label{prop:univ}
The partition $(1^{m_1},2^{m_2},\ldots, \ell^{m_\ell})$ is $t$-universal if and only if
\[ \sum_{j=1}^s jm_j \ge s \] 
for all $s \in \{1,\ldots, t\}$.
\end{proposition}

\begin{proof}
The condition is obviously necessary. Suppose that it holds.
Let $s \le t$ be given. By hypothesis we have $\sum_{j=1}^s jm_j \ge s$.
Let $q$ be greatest such that $q \le s$ and $m_q \not=0$.
Let $m'_j = m_j$ if $j \not= q$ and let $m'_q = m_q-1$.
We consider two cases.
\begin{itemize}
\item[(i)] Suppose $s-q \ge q$. We first show that the partition $(1^{m_1'}, \ldots, q^{m_q'})$
is $(s-q)$-universal. Let $u \le s-q$ be given. If $u < q$ then $\sum_{j=1}^u jm_j' = \sum_{j=1}^u
jm_j \ge u$. If $u \ge q$ then we have
\[ \sum_{j=1}^u jm_j' = \sum_{j=1}^q jm_j' = \sum_{j=1}^q jm_j - q = \sum_{j=1}^s jm_j - q \ge s -q
\ge u. \]
Hence, by induction, 
$(1^{m_1'},\ldots,q^{m_q'})$ is $(s-q)$-universal. In particular it has
a subpartition of size $s-q$. Since $(1^{m_1},\ldots,q^{m_q})$ has an extra part of size~$q$,
it has a subpartition of size $s$.

\item[(ii)] If $s-q < q$ then
$\sum_{j=1}^{u} jm_j' = \sum_{j=1}^{u} jm_j$ for any $u \le s-q$. 
By hypothesis $\sum_{j=1}^{u} jm_j \ge u$, so by
induction the partition
$(1^{m_1'},\ldots,q^{m_q'})$ is $(s-q)$-universal. The proof finishes as in (i).
\hfill$\qedhere$
\end{itemize}
\end{proof}

In Step (A2) of the algorithm we test whether each partial row is $k$-universal using the criterion
in Proposition~\ref{prop:univ}.
Any partial row that passes this test can immediately be discarded.

\subsection*{Computation of $i(\infty,k)$ given the table}

Let $p(\infty,k) = 1-i(\infty ,k)$.
Let $r = (m_1,\ldots,m_k)$ be a row of the table.  For each $j$,
define 
\[ x_j(r) = \begin{cases} \displaystyle \e^{-1/j}\frac{1}{j^{m_j} m_j!} & \text{if $m_j < \lfloor k/j \rfloor$}
\\ \displaystyle 1 -  \e^{-1/j} \sum_{0 \le i < \lfloor k/j\rfloor} 
\frac{1}{j^i i!} & \text{if $m_j = \lfloor k/j \rfloor$.}\end{cases}
\]

\begin{lemma}\label{lemma:ink}
For each $k \in \N$, the limiting probability $p(\infty,k)$ is equal to the sum of $x_1(r)\ldots x_k(r)$ over every
row of the table produced by the Derangements Table Algorithm with input $k$.
\end{lemma}

\begin{proof}
Let $r = (m_1,\ldots,m_k)$ be a row of the table.
Let $J = \{j \in \{1,\ldots,k\} : m_j = \lfloor k/j \rfloor\}$. Note that
any partition $(1^{m_1'},\ldots, k^{m_k'})$ such that $m'_i = m_i$ if $i \not\in J$
is $k$-free. Each partition of this form with $m'_j < m_j$ for some $j \in J$ corresponds
to a row appearing later in the table.
Thus $r$ must account precisely for the partitions with \emph{at least} $\lfloor k/j\rfloor$
parts of size $j$ for every $j \in J$, and with \emph{exactly} $m_i$ parts of size $i$ for every $i \not\in J$.
The correct contribution
from $r$ to $p(\infty,k)$ is therefore $x_1(r)\ldots, x_k(r)$.
\end{proof}

\subsection*{Example}

The Derangements Table Algorithm can readily be implemented by hand for small values of~$k$. 
As an illustration, the table for $k=4$ is shown below.

\begin{center}
\begin{tabular}{ccccll} \\ \toprule 
1 & 2 & 3 & 4 & \multicolumn{2}{l}{probability} \\ \midrule
3 & 0 & 0 & 0 & $ \mfrac{1}{6}\e^{-25/12}$ & 0.020752\\
2 & 0 & 0 & 0 & $ \mfrac{1}{2}\e^{-25/12}$ & 0.062257 \\ 
1 & 1 & 0 & 0 & $ \mfrac{1}{2}\e^{-25/12}$ & 0.062257 \\ 
1 & 0 & 0 & 0 & $ \e^{-25/12}$ & 0.124514 \\
0 & 1 & 1 & 0 & $\mfrac{1}{2} \e^{-7/4} \bigl( 1 - \e^{-1/3} \bigr)$ & 0.024630 \\
0 & 1 & 0 & 0 & $\mfrac{1}{2}\e^{-25/12} $ & 0.062257 \\
0 & 0 & 1 & 0 & $\e^{-7/4} \bigl( 1 - \e^{-1/3} \bigr) $ & 0.049259 \\ 
0 & 0 & 0 & 0 & $\e^{-25/12}$ & 0.124514 \\ \bottomrule
\end{tabular}
\end{center}


\medskip
\noindent For example,
the exact probability for 
the row $(0,1,1,0)$, accounting for all partitions of the form $(2,3^a)$ with $a \ge 1$, 
is $\e^{-1} \e^{-1/2} \mfrac{1}{2} \bigl( 1 - \e^{-1/3} \bigr) \e^{-1/4}$. 
The only other row having a multiplicity $m_j$ such that $m_j = \lfloor k /j \rfloor$ is
$(0,0,1,0)$. Thus all remaining rows
contribute a rational multiple of $\e^{-1-1/2-1/3-1/4} = \e^{-25/12}$ to the
limiting probability. One finds that
\[ p(\infty,4) = \mfrac{3}{2} (1-\e^{-1/3})\,\e^{-7/4}+\mfrac{11}{3} \e^{-25/12}
\approx 0.530442. \]

\section*{Appendix A: Haskell implementation} 

The Derangements Table Algorithm has been implemented in Haskell \cite{Haskell98}. The arXiv submission
of this paper includes the relevant files: \texttt{DerangementsTable.hs} and \texttt{Main.hs}.

We note two refinements to the basic version of the algorithm presented above.

\begin{itemize}
\item[(1)] In any row the $k$-th element, corresponding to cycles of length $k$, is always zero.
It therefore suffices to work with rows and partial rows of length at most $k-1$, scaling
by $\exp(-1/k)$ to account for the $k$-cycles.\\[-6pt]
 
\item[(2)] Suppose that $r$ is a partial row of length $\ell$ 
and $d \in \N$ is such that (i) $n$ is not divisible by $d$, and (ii)
for each $j \in \{1,\ldots, \ell\}$, either $j$ is divisible by $d$, or 
$a_j = 0$. Then $r$ is clearly $k$-free. Applying this trick to the partial
rows considered in Step A2 when finding $m$
gives a surprisingly large speed-up. For example, when $k=25$,
it reduces the running time from approximately $44$ minutes to approximately $12$ minutes.
\end{itemize}

\smallskip
If a partial row is neither $k$-universal, nor meets the condition for the divisibility check,
then an exhaustive search is made for subpartitions of size $k$ using the function
\hf{subpartitionSizes}. This accounts for the majority of the running time.
For example, when $k=25$, 5\hskip1pt240\hskip1pt351 partial rows are considered;
of these 103\hskip1pt189 are $k$-universal and $2\hskip1pt041\hskip1pt735$ are ruled out by divisibility, leaving 
$3\hskip1pt095\hskip1pt427$
on which a full test must be made. The final table has $2\,235\,240$ rows.

The values shown in Appendix~B take about 72 hours to compute on one core of a 2.7GHz Intel i5.

The key functions in \texttt{DerangementsTable.hs} are reproduced below.

\renewcommand{\c}[1]{{\color{blue}#1}}
\bigskip
\begin{minipage}{5in}\small
\begin{Verbatim}[commandchars=\\\{\}]
universality ms = \c{case} as' \c{of} [] -> 0
                              otherwise -> s
    \c{where} ss = zip [1..] $ partialSums [a*m | (a,m) <- zip [1..] ms]
          (as', _)   = span (uncurry (<=)) ss 
          (a, s) : _ = reverse as'

partialSums = scanl1 (+)  

subpartition k ms 
    | universality ms >= k    = True
    | divisibilityTest k ms   = False
    | otherwise = k `elem` (subpartitionSizes $ zip [1..] ms)

divisibilityTest k ms = 
    \c{let} test c = k `mod` c /= 0
                 && and [j `mod` c == 0 || m == 0 | (j,m) <- zip [1..] ms]
    \c{in}  or [test c | c <- [2..k `div` 2]]

subpartitionSizes [] = [0]
subpartitionSizes ((a, m) : ams) 
    = [l*a + y | l <- [0..m], y <- subpartitionSizes ams]
\end{Verbatim}
\end{minipage}

\bigskip
Clearly the \texttt{subpartition} function can also be used to compute the finite
derangement probabilities $i(n,k)$ and $p(n,k) = 1-i(n,k)$. 
The values shown in Appendices B and C take about 24 hours to compute on one core of
a 2.7 GHz Intel~i5.
The relevant code is reproduced below.

\bigskip
\begin{minipage}{5in}\small                
\begin{Verbatim}[commandchars=\\\{\}]
fixesSet k ms | universality ms >= k = True
              | otherwise = subpartition k ms

centralizerSize ms = product [f (x, m) | (x,m) <- zip [1..] ms]
    \c{where} f (x, m) = \c{let} x' = toInteger x
                         m' = toInteger m
                     \c{in} x'^m' * factorial m'

finiteTable n k = [(ms, subpartition k ms, centralizerSize ms)
                                   | xs <- partitions n, 
                                     \c{let} ms = toMultiplicities n xs]

sumStrict = foldl1' (+)

i n k = sumStrict [1 \% c | (_, True, c) <- finiteTable n k]
\end{Verbatim}
\end{minipage}

\newpage
\section*{Appendix B: $i(\infty, k)$ for $k \le 30$}

In the table below probabilities are rounded to $8$ decimal places. 
These probabilities were computed using {\sc Mathematica} to add up the contributions
given by Lemma~\ref{lemma:ink}
from each row of the table produced by the Derangements Table Algorithm, requiring an exact answer
to 20 decimal places. The number of rows of the table is also given.

\bigskip
\begin{center}
\begin{tabular}{rclrcl} \toprule
$k$ & $i(\infty,k)$  & $\mathrm{rows}(k)$ & $k$ & $i(\infty,k)$  & $\mathrm{rows}(k)$    \\ \midrule
1 & 0.63212056 & 1 & 16 & 0.33807249 & 8420 \\[-1.00pt]
2 & 0.55373968 & 2 & 17 & 0.33297333 & 19553 \\[-1.00pt]
3 & 0.49658324 & 4 & 18 & 0.32907588 & 23586 \\[-1.00pt]
4 & 0.46955773 & 8 & 19 & 0.32472908 & 61470 \\[-1.00pt]
5 & 0.44145770 & 15 & 20 & 0.32132422 & 71413 \\[-1.00pt]
6 & 0.42505870 & 29 & 21 & 0.31750065 & 193303 \\[-1.00pt]
7 & 0.40848113 & 53 & 22 & 0.31449862 & 216928 \\[-1.00pt]
8 & 0.39727771 & 93 & 23 & 0.31110428 & 508502 \\[-1.00pt]
9 & 0.38516443 & 187 & 24 & 0.30842280 & 532542 \\[-1.00pt]
10 & 0.37687192 & 305 &25 & 0.30538904 & 2235240 \\[-1.00pt]
11 & 0.36773064 & 561 & 26 & 0.30295361 & 1817364 \\[-1.00pt]
12 & 0.36119415 & 916 & 27 & 0.30021508 & 5143197 \\[-1.00pt]
13 & 0.35396068 & 2067 & 28 & 0.29801340 & 4961040 \\[-1.00pt]
14 & 0.34855007 & 2782 & 29 & 0.29550915 & 17517544 \\[-1.00pt]
15 & 0.34256331 & 5670 & 30 & 0.29348611 & 12022223 \\
\bottomrule
\end{tabular}
\end{center}

\bigskip
Question 1 in \cite{EberhardFordGreen} asks if there is a constant 
$C$ such that $i(\infty,k) \sim Ck^{-\delta}(\log k)^{-3/2}$. The ratio
of the two sides is shown below for $1 \le k \le 30$.
If Question 1
has an affirmative answer then the ratio converges to $C$ as $k \rightarrow \infty$.

\enlargethispage{36pt}
\begin{figure}[h]
\begin{center}
\hspace*{-0.15in}\scalebox{1.10}{\includegraphics{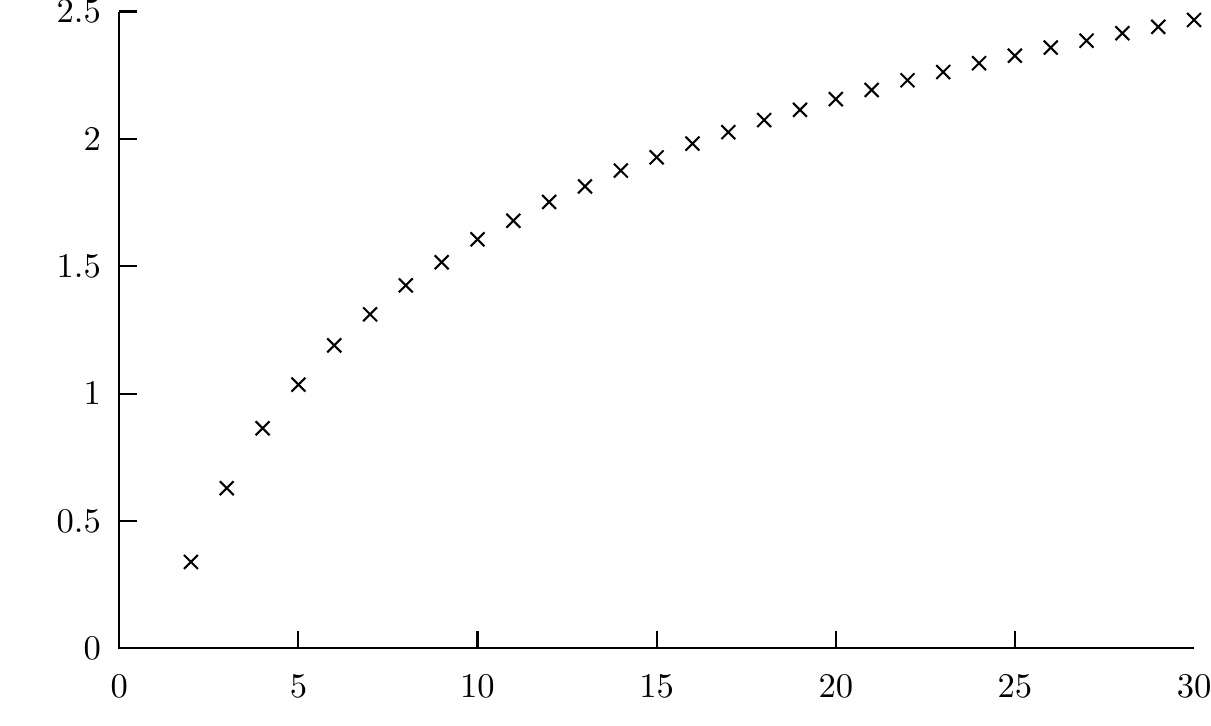}}
\caption{$i(\infty,k)/k^{-\delta}(\log k)^{-3/2}$ for $2 \le k \le 30$}
\end{center}
\end{figure}


\newgeometry{margin=6pt}
\begin{landscape}
\section*{Appendix C: $i(n,k)$ for $n \le 70$ and $i(\infty,k)$
(values are rounded to 5 decimal places)}

\begin{center}
\scalebox{0.52}{
\renewcommand{\sp}{.}
\begin{tabular}{llllllllllllllllllllllllllllllllllll}\toprule 
$n \backslash k$ & 1 & 2 & 3 & 4 & 5 & 6 & 7 & 8 & 9 & 10 & 11 & 12 & 13 & 14 & 15 & 16 & 17 & 18 & 19 & 20 & 21 & 22 & 23 & 24 & 25 & 26 & 27 & 28 & 29 & 30 & 31 & 32 & 33 & 34 & 35 \\ \midrule
2 & .50000 & \sp & \sp & \sp & \sp & \sp & \sp & \sp & \sp & \sp & \sp & \sp & \sp & \sp & \sp & \sp & \sp & \sp & \sp & \sp & \sp & \sp & \sp & \sp & \sp & \sp & \sp & \sp & \sp & \sp & \sp & \sp & \sp & \sp & \sp \\
3 & .66667 & \sp & \sp & \sp & \sp & \sp & \sp & \sp & \sp & \sp & \sp & \sp & \sp & \sp & \sp & \sp & \sp & \sp & \sp & \sp & \sp & \sp & \sp & \sp & \sp & \sp & \sp & \sp & \sp & \sp & \sp & \sp & \sp & \sp & \sp \\
4 & .62500 & .41667 & \sp & \sp & \sp & \sp & \sp & \sp & \sp & \sp & \sp & \sp & \sp & \sp & \sp & \sp & \sp & \sp & \sp & \sp & \sp & \sp & \sp & \sp & \sp & \sp & \sp & \sp & \sp & \sp & \sp & \sp & \sp & \sp & \sp \\
5 & .63333 & .55000 & \sp & \sp & \sp & \sp & \sp & \sp & \sp & \sp & \sp & \sp & \sp & \sp & \sp & \sp & \sp & \sp & \sp & \sp & \sp & \sp & \sp & \sp & \sp & \sp & \sp & \sp & \sp & \sp & \sp & \sp & \sp & \sp & \sp \\
6 & .63194 & .57778 & .36250 & \sp & \sp & \sp & \sp & \sp & \sp & \sp & \sp & \sp & \sp & \sp & \sp & \sp & \sp & \sp & \sp & \sp & \sp & \sp & \sp & \sp & \sp & \sp & \sp & \sp & \sp & \sp & \sp & \sp & \sp & \sp & \sp \\
7 & .63214 & .55159 & .49048 & \sp & \sp & \sp & \sp & \sp & \sp & \sp & \sp & \sp & \sp & \sp & \sp & \sp & \sp & \sp & \sp & \sp & \sp & \sp & \sp & \sp & \sp & \sp & \sp & \sp & \sp & \sp & \sp & \sp & \sp & \sp & \sp \\
8 & .63212 & .55089 & .50037 & .33770 & \sp & \sp & \sp & \sp & \sp & \sp & \sp & \sp & \sp & \sp & \sp & \sp & \sp & \sp & \sp & \sp & \sp & \sp & \sp & \sp & \sp & \sp & \sp & \sp & \sp & \sp & \sp & \sp & \sp & \sp & \sp \\
9 & .63212 & .55424 & .51478 & .44819 & \sp & \sp & \sp & \sp & \sp & \sp & \sp & \sp & \sp & \sp & \sp & \sp & \sp & \sp & \sp & \sp & \sp & \sp & \sp & \sp & \sp & \sp & \sp & \sp & \sp & \sp & \sp & \sp & \sp & \sp & \sp \\
10 & .63212 & .55399 & .49467 & .46770 & .31321 & \sp & \sp & \sp & \sp & \sp & \sp & \sp & \sp & \sp & \sp & \sp & \sp & \sp & \sp & \sp & \sp & \sp & \sp & \sp & \sp & \sp & \sp & \sp & \sp & \sp & \sp & \sp & \sp & \sp & \sp \\
11 & .63212 & .55367 & .49524 & .47767 & .41971 & \sp & \sp & \sp & \sp & \sp & \sp & \sp & \sp & \sp & \sp & \sp & \sp & \sp & \sp & \sp & \sp & \sp & \sp & \sp & \sp & \sp & \sp & \sp & \sp & \sp & \sp & \sp & \sp & \sp & \sp \\
12 & .63212 & .55372 & .49505 & .48527 & .43450 & .29877 & \sp & \sp & \sp & \sp & \sp & \sp & \sp & \sp & \sp & \sp & \sp & \sp & \sp & \sp & \sp & \sp & \sp & \sp & \sp & \sp & \sp & \sp & \sp & \sp & \sp & \sp & \sp & \sp & \sp \\
13 & .63212 & .55375 & .49700 & .46991 & .44605 & .39720 & \sp & \sp & \sp & \sp & \sp & \sp & \sp & \sp & \sp & \sp & \sp & \sp & \sp & \sp & \sp & \sp & \sp & \sp & \sp & \sp & \sp & \sp & \sp & \sp & \sp & \sp & \sp & \sp & \sp \\
14 & .63212 & .55374 & .49673 & .46796 & .44905 & .41645 & .28503 & \sp & \sp & \sp & \sp & \sp & \sp & \sp & \sp & \sp & \sp & \sp & \sp & \sp & \sp & \sp & \sp & \sp & \sp & \sp & \sp & \sp & \sp & \sp & \sp & \sp & \sp & \sp & \sp \\
15 & .63212 & .55374 & .49667 & .46821 & .45466 & .42313 & .37927 & \sp & \sp & \sp & \sp & \sp & \sp & \sp & \sp & \sp & \sp & \sp & \sp & \sp & \sp & \sp & \sp & \sp & \sp & \sp & \sp & \sp & \sp & \sp & \sp & \sp & \sp & \sp & \sp \\
16 & .63212 & .55374 & .49654 & .46849 & .44150 & .43054 & .39499 & .27552 & \sp & \sp & \sp & \sp & \sp & \sp & \sp & \sp & \sp & \sp & \sp & \sp & \sp & \sp & \sp & \sp & \sp & \sp & \sp & \sp & \sp & \sp & \sp & \sp & \sp & \sp & \sp \\
17 & .63212 & .55374 & .49657 & .46976 & .44075 & .43301 & .40626 & .36430 & \sp & \sp & \sp & \sp & \sp & \sp & \sp & \sp & \sp & \sp & \sp & \sp & \sp & \sp & \sp & \sp & \sp & \sp & \sp & \sp & \sp & \sp & \sp & \sp & \sp & \sp & \sp \\
18 & .63212 & .55374 & .49658 & .46979 & .44001 & .43691 & .40899 & .38199 & .26568 & \sp & \sp & \sp & \sp & \sp & \sp & \sp & \sp & \sp & \sp & \sp & \sp & \sp & \sp & \sp & \sp & \sp & \sp & \sp & \sp & \sp & \sp & \sp & \sp & \sp & \sp \\
19 & .63212 & .55374 & .49659 & .46966 & .44056 & .42594 & .41518 & .39097 & .35199 & \sp & \sp & \sp & \sp & \sp & \sp & \sp & \sp & \sp & \sp & \sp & \sp & \sp & \sp & \sp & \sp & \sp & \sp & \sp & \sp & \sp & \sp & \sp & \sp & \sp & \sp \\
20 & .63212 & .55374 & .49658 & .46960 & .44073 & .42425 & .41586 & .39653 & .36734 & .25881 & \sp & \sp & \sp & \sp & \sp & \sp & \sp & \sp & \sp & \sp & \sp & \sp & \sp & \sp & \sp & \sp & \sp & \sp & \sp & \sp & \sp & \sp & \sp & \sp & \sp \\
21 & .63212 & .55374 & .49658 & .46953 & .44167 & .42416 & .41929 & .40026 & .37846 & .34102 & \sp & \sp & \sp & \sp & \sp & \sp & \sp & \sp & \sp & \sp & \sp & \sp & \sp & \sp & \sp & \sp & \sp & \sp & \sp & \sp & \sp & \sp & \sp & \sp & \sp \\
22 & .63212 & .55374 & .49658 & .46954 & .44160 & .42391 & .40938 & .40391 & .38191 & .35779 & .25153 & \sp & \sp & \sp & \sp & \sp & \sp & \sp & \sp & \sp & \sp & \sp & \sp & \sp & \sp & \sp & \sp & \sp & \sp & \sp & \sp & \sp & \sp & \sp & \sp \\
23 & .63212 & .55374 & .49658 & .46955 & .44159 & .42429 & .40825 & .40472 & .38708 & .36731 & .33177 & \sp & \sp & \sp & \sp & \sp & \sp & \sp & \sp & \sp & \sp & \sp & \sp & \sp & \sp & \sp & \sp & \sp & \sp & \sp & \sp & \sp & \sp & \sp & \sp \\
24 & .63212 & .55374 & .49658 & .46956 & .44150 & .42450 & .40749 & .40726 & .38869 & .37262 & .34677 & .24618 & \sp & \sp & \sp & \sp & \sp & \sp & \sp & \sp & \sp & \sp & \sp & \sp & \sp & \sp & \sp & \sp & \sp & \sp & \sp & \sp & \sp & \sp & \sp \\
25 & .63212 & .55374 & .49658 & .46956 & .44147 & .42519 & .40766 & .39862 & .39190 & .37580 & .35747 & .32342 & \sp & \sp & \sp & \sp & \sp & \sp & \sp & \sp & \sp & \sp & \sp & \sp & \sp & \sp & \sp & \sp & \sp & \sp & \sp & \sp & \sp & \sp & \sp \\
26 & .63212 & .55374 & .49658 & .46956 & .44143 & .42523 & .40747 & .39712 & .39199 & .37970 & .36171 & .33935 & .24052 & \sp & \sp & \sp & \sp & \sp & \sp & \sp & \sp & \sp & \sp & \sp & \sp & \sp & \sp & \sp & \sp & \sp & \sp & \sp & \sp & \sp & \sp \\
27 & .63212 & .55374 & .49658 & .46956 & .44145 & .42517 & .40788 & .39663 & .39434 & .38117 & .36636 & .34861 & .31602 & \sp & \sp & \sp & \sp & \sp & \sp & \sp & \sp & \sp & \sp & \sp & \sp & \sp & \sp & \sp & \sp & \sp & \sp & \sp & \sp & \sp & \sp \\
28 & .63212 & .55374 & .49658 & .46956 & .44145 & .42514 & .40802 & .39633 & .38644 & .38346 & .36803 & .35458 & .33061 & .23617 & \sp & \sp & \sp & \sp & \sp & \sp & \sp & \sp & \sp & \sp & \sp & \sp & \sp & \sp & \sp & \sp & \sp & \sp & \sp & \sp & \sp \\
29 & .63212 & .55374 & .49658 & .46956 & .44146 & .42509 & .40861 & .39646 & .38530 & .38365 & .37149 & .35780 & .34092 & .30931 & \sp & \sp & \sp & \sp & \sp & \sp & \sp & \sp & \sp & \sp & \sp & \sp & \sp & \sp & \sp & \sp & \sp & \sp & \sp & \sp & \sp \\
30 & .63212 & .55374 & .49658 & .46956 & .44146 & .42507 & .40861 & .39644 & .38452 & .38548 & .37210 & .36052 & .34544 & .32452 & .23153 & \sp & \sp & \sp & \sp & \sp & \sp & \sp & \sp & \sp & \sp & \sp & \sp & \sp & \sp & \sp & \sp & \sp & \sp & \sp & \sp \\
31 & .63212 & .55374 & .49658 & .46956 & .44146 & .42504 & .40861 & .39675 & .38448 & .37845 & .37422 & .36279 & .35022 & .33370 & .30330 & \sp & \sp & \sp & \sp & \sp & \sp & \sp & \sp & \sp & \sp & \sp & \sp & \sp & \sp & \sp & \sp & \sp & \sp & \sp & \sp \\
32 & .63212 & .55374 & .49658 & .46956 & .44146 & .42505 & .40856 & .39691 & .38426 & .37702 & .37409 & .36522 & .35202 & .33936 & .31741 & .22795 & \sp & \sp & \sp & \sp & \sp & \sp & \sp & \sp & \sp & \sp & \sp & \sp & \sp & \sp & \sp & \sp & \sp & \sp & \sp \\
33 & .63212 & .55374 & .49658 & .46956 & .44146 & .42505 & .40854 & .39737 & .38447 & .37650 & .37584 & .36592 & .35482 & .34318 & .32725 & .29770 & \sp & \sp & \sp & \sp & \sp & \sp & \sp & \sp & \sp & \sp & \sp & \sp & \sp & \sp & \sp & \sp & \sp & \sp & \sp \\
34 & .63212 & .55374 & .49658 & .46956 & .44146 & .42506 & .40850 & .39741 & .38445 & .37618 & .36917 & .36749 & .35613 & .34599 & .33202 & .31231 & .22404 & \sp & \sp & \sp & \sp & \sp & \sp & \sp & \sp & \sp & \sp & \sp & \sp & \sp & \sp & \sp & \sp & \sp & \sp \\
35 & .63212 & .55374 & .49658 & .46956 & .44146 & .42506 & .40848 & .39739 & .38475 & .37614 & .36813 & .36745 & .35837 & .34781 & .33690 & .32139 & .29267 & \sp & \sp & \sp & \sp & \sp & \sp & \sp & \sp & \sp & \sp & \sp & \sp & \sp & \sp & \sp & \sp & \sp & \sp \\
36 & .63212 & .55374 & .49658 & .46956 & .44146 & .42506 & .40846 & .39737 & .38485 & .37606 & .36736 & .36890 & .35861 & .35009 & .33879 & .32688 & .30634 & .22096 & \sp & \sp & \sp & \sp & \sp & \sp & \sp & \sp & \sp & \sp & \sp & \sp & \sp & \sp & \sp & \sp & \sp \\
37 & .63212 & .55374 & .49658 & .46956 & .44146 & .42506 & .40847 & .39733 & .38525 & .37623 & .36723 & .36286 & .36016 & .35137 & .34156 & .33091 & .31594 & .28794 & \sp & \sp & \sp & \sp & \sp & \sp & \sp & \sp & \sp & \sp & \sp & \sp & \sp & \sp & \sp & \sp & \sp \\
38 & .63212 & .55374 & .49658 & .46956 & .44146 & .42506 & .40847 & .39732 & .38527 & .37626 & .36699 & .36159 & .35988 & .35309 & .34267 & .33383 & .32067 & .30201 & .21765 & \sp & \sp & \sp & \sp & \sp & \sp & \sp & \sp & \sp & \sp & \sp & \sp & \sp & \sp & \sp & \sp \\
39 & .63212 & .55374 & .49658 & .46956 & .44146 & .42506 & .40848 & .39729 & .38528 & .37649 & .36706 & .36105 & .36130 & .35337 & .34491 & .33570 & .32550 & .31089 & .28361 & \sp & \sp & \sp & \sp & \sp & \sp & \sp & \sp & \sp & \sp & \sp & \sp & \sp & \sp & \sp & \sp \\
40 & .63212 & .55374 & .49658 & .46956 & .44146 & .42506 & .40848 & .39728 & .38525 & .37661 & .36698 & .36068 & .35558 & .35453 & .34560 & .33762 & .32767 & .31625 & .29691 & .21498 & \sp & \sp & \sp & \sp & \sp & \sp & \sp & \sp & \sp & \sp & \sp & \sp & \sp & \sp & \sp \\
41 & .63212 & .55374 & .49658 & .46956 & .44146 & .42506 & .40848 & .39726 & .38524 & .37694 & .36718 & .36060 & .35454 & .35439 & .34722 & .33912 & .33041 & .32038 & .30627 & .27954 & \sp & \sp & \sp & \sp & \sp & \sp & \sp & \sp & \sp & \sp & \sp & \sp & \sp & \sp & \sp \\
42 & .63212 & .55374 & .49658 & .46956 & .44146 & .42506 & .40848 & .39727 & .38520 & .37698 & .36719 & .36050 & .35381 & .35558 & .34726 & .34072 & .33155 & .32341 & .31093 & .29312 & .21209 & \sp & \sp & \sp & \sp & \sp & \sp & \sp & \sp & \sp & \sp & \sp & \sp & \sp & \sp \\
43 & .63212 & .55374 & .49658 & .46956 & .44146 & .42506 & .40848 & .39727 & .38519 & .37697 & .36741 & .36054 & .35363 & .35029 & .34843 & .34150 & .33354 & .32539 & .31569 & .30186 & .27578 & \sp & \sp & \sp & \sp & \sp & \sp & \sp & \sp & \sp & \sp & \sp & \sp & \sp & \sp \\
44 & .63212 & .55374 & .49658 & .46956 & .44146 & .42506 & .40848 & .39727 & .38517 & .37696 & .36750 & .36052 & .35336 & .34915 & .34812 & .34274 & .33450 & .32723 & .31810 & .30711 & .28871 & .20976 & \sp & \sp & \sp & \sp & \sp & \sp & \sp & \sp & \sp & \sp & \sp & \sp & \sp \\
45 & .63212 & .55374 & .49658 & .46956 & .44146 & .42506 & .40848 & .39728 & .38516 & .37694 & .36780 & .36068 & .35338 & .34859 & .34929 & .34285 & .33604 & .32845 & .32079 & .31122 & .29780 & .27220 & \sp & \sp & \sp & \sp & \sp & \sp & \sp & \sp & \sp & \sp & \sp & \sp & \sp \\
46 & .63212 & .55374 & .49658 & .46956 & .44146 & .42506 & .40848 & .39728 & .38515 & .37693 & .36782 & .36072 & .35327 & .34820 & .34425 & .34378 & .33645 & .33016 & .32203 & .31426 & .30248 & .28539 & .20720 & \sp & \sp & \sp & \sp & \sp & \sp & \sp & \sp & \sp & \sp & \sp & \sp \\
47 & .63212 & .55374 & .49658 & .46956 & .44146 & .42506 & .40848 & .39728 & .38515 & .37690 & .36783 & .36090 & .35336 & .34808 & .34329 & .34355 & .33770 & .33111 & .32394 & .31644 & .30712 & .29397 & .26890 & \sp & \sp & \sp & \sp & \sp & \sp & \sp & \sp & \sp & \sp & \sp & \sp \\
48 & .63212 & .55374 & .49658 & .46956 & .44146 & .42506 & .40848 & .39728 & .38516 & .37689 & .36781 & .36099 & .35335 & .34795 & .34261 & .34456 & .33762 & .33229 & .32474 & .31831 & .30958 & .29907 & .28149 & .20513 & \sp & \sp & \sp & \sp & \sp & \sp & \sp & \sp & \sp & \sp & \sp \\
49 & .63212 & .55374 & .49658 & .46956 & .44146 & .42506 & .40848 & .39728 & .38516 & .37688 & .36781 & .36125 & .35350 & .34795 & .34236 & .33987 & .33854 & .33277 & .32642 & .31956 & .31243 & .30320 & .29035 & .26573 & \sp & \sp & \sp & \sp & \sp & \sp & \sp & \sp & \sp & \sp & \sp \\
50 & .63212 & .55374 & .49658 & .46956 & .44146 & .42506 & .40848 & .39728 & .38516 & .37687 & .36778 & .36129 & .35353 & .34791 & .34208 & .33881 & .33822 & .33375 & .32700 & .32093 & .31371 & .30620 & .29500 & .27855 & .20286 & \sp & \sp & \sp & \sp & \sp & \sp & \sp & \sp & \sp & \sp \\
51 & .63212 & .55374 & .49658 & .46956 & .44146 & .42506 & .40848 & .39728 & .38516 & .37686 & .36777 & .36128 & .35371 & .34798 & .34207 & .33825 & .33922 & .33375 & .32818 & .32202 & .31556 & .30850 & .29954 & .28693 & .26278 & \sp & \sp & \sp & \sp & \sp & \sp & \sp & \sp & \sp & \sp \\
52 & .63212 & .55374 & .49658 & .46956 & .44146 & .42506 & .40848 & .39728 & .38516 & .37686 & .36775 & .36128 & .35378 & .34799 & .34195 & .33789 & .33471 & .33449 & .32844 & .32330 & .31641 & .31039 & .30205 & .29199 & .27508 & .20099 & \sp & \sp & \sp & \sp & \sp & \sp & \sp & \sp & \sp \\
53 & .63212 & .55374 & .49658 & .46956 & .44146 & .42506 & .40848 & .39728 & .38517 & .37687 & .36775 & .36127 & .35401 & .34813 & .34199 & .33771 & .33381 & .33424 & .32941 & .32395 & .31785 & .31176 & .30488 & .29606 & .28374 & .25995 & \sp & \sp & \sp & \sp & \sp & \sp & \sp & \sp & \sp \\
54 & .63212 & .55374 & .49658 & .46956 & .44146 & .42506 & .40848 & .39728 & .38517 & .37687 & .36773 & .36126 & .35404 & .34817 & .34195 & .33756 & .33315 & .33511 & .32927 & .32485 & .31857 & .31308 & .30632 & .29900 & .28833 & .27242 & .19895 & \sp & \sp & \sp & \sp & \sp & \sp & \sp & \sp \\
55 & .63212 & .55374 & .49658 & .46956 & .44146 & .42506 & .40848 & .39728 & .38517 & .37687 & .36773 & .36124 & .35405 & .34831 & .34204 & .33754 & .33289 & .33089 & .33005 & .32515 & .31984 & .31398 & .30815 & .30139 & .29278 & .28066 & .25730 & \sp & \sp & \sp & \sp & \sp & \sp & \sp & \sp \\
56 & .63212 & .55374 & .49658 & .46956 & .44146 & .42506 & .40848 & .39728 & .38516 & .37687 & .36772 & .36123 & .35404 & .34839 & .34205 & .33747 & .33262 & .32991 & .32970 & .32595 & .32023 & .31525 & .30906 & .30339 & .29533 & .28561 & .26932 & .19727 & \sp & \sp & \sp & \sp & \sp & \sp & \sp \\
57 & .63212 & .55374 & .49658 & .46956 & .44146 & .42506 & .40848 & .39728 & .38516 & .37687 & .36772 & .36121 & .35404 & .34859 & .34218 & .33751 & .33256 & .32937 & .33058 & .32588 & .32116 & .31601 & .31044 & .30474 & .29811 & .28966 & .27776 & .25474 & \sp & \sp & \sp & \sp & \sp & \sp & \sp \\
58 & .63212 & .55374 & .49658 & .46956 & .44146 & .42506 & .40848 & .39728 & .38516 & .37687 & .36772 & .36120 & .35402 & .34863 & .34221 & .33750 & .33242 & .32901 & .32650 & .32650 & .32130 & .31702 & .31106 & .30609 & .29967 & .29256 & .28234 & .26691 & .19541 & \sp & \sp & \sp & \sp & \sp & \sp \\
59 & .63212 & .55374 & .49658 & .46956 & .44146 & .42506 & .40848 & .39728 & .38516 & .37687 & .36773 & .36119 & .35401 & .34863 & .34236 & .33758 & .33245 & .32882 & .32564 & .32623 & .32210 & .31745 & .31236 & .30703 & .30155 & .29497 & .28671 & .27500 & .25233 & \sp & \sp & \sp & \sp & \sp & \sp \\
60 & .63212 & .55374 & .49658 & .46956 & .44146 & .42506 & .40848 & .39728 & .38516 & .37687 & .36773 & .36119 & .35400 & .34863 & .34242 & .33760 & .33239 & .32866 & .32502 & .32700 & .32193 & .31818 & .31285 & .30806 & .30246 & .29695 & .28923 & .27986 & .26409 & .19387 & \sp & \sp & \sp & \sp & \sp \\
61 & .63212 & .55374 & .49658 & .46956 & .44146 & .42506 & .40848 & .39728 & .38516 & .37687 & .36773 & .36118 & .35399 & .34862 & .34261 & .33772 & .33245 & .32861 & .32476 & .32315 & .32258 & .31838 & .31385 & .30891 & .30384 & .29844 & .29202 & .28388 & .27237 & .25001 & \sp & \sp & \sp & \sp & \sp \\
62 & .63212 & .55374 & .49658 & .46956 & .44146 & .42506 & .40848 & .39728 & .38516 & .37687 & .36773 & .36119 & .35397 & .34861 & .34263 & .33776 & .33244 & .32854 & .32448 & .32224 & .32224 & .31904 & .31412 & .30992 & .30449 & .29980 & .29362 & .28674 & .27689 & .26190 & .19219 & \sp & \sp & \sp & \sp \\
63 & .63212 & .55374 & .49658 & .46956 & .44146 & .42506 & .40848 & .39728 & .38516 & .37687 & .36773 & .36119 & .35397 & .34860 & .34265 & .33788 & .33252 & .32855 & .32440 & .32172 & .32302 & .31893 & .31487 & .31047 & .30561 & .30077 & .29555 & .28915 & .28119 & .26984 & .24781 & \sp & \sp & \sp & \sp \\
64 & .63212 & .55374 & .49658 & .46956 & .44146 & .42506 & .40848 & .39728 & .38516 & .37687 & .36773 & .36119 & .35396 & .34859 & .34264 & .33794 & .33254 & .32852 & .32427 & .32136 & .31928 & .31946 & .31493 & .31127 & .30621 & .30178 & .29652 & .29114 & .28371 & .27462 & .25934 & .19077 & \sp & \sp & \sp \\
65 & .63212 & .55374 & .49658 & .46956 & .44146 & .42506 & .40848 & .39728 & .38516 & .37687 & .36773 & .36119 & .35396 & .34858 & .34264 & .33811 & .33265 & .32857 & .32426 & .32117 & .31848 & .31918 & .31561 & .31158 & .30722 & .30251 & .29788 & .29271 & .28646 & .27862 & .26745 & .24569 & \sp & \sp & \sp \\
66 & .63212 & .55374 & .49658 & .46956 & .44146 & .42506 & .40848 & .39728 & .38516 & .37687 & .36773 & .36119 & .35395 & .34857 & .34262 & .33814 & .33268 & .32858 & .32418 & .32100 & .31789 & .31987 & .31542 & .31218 & .30757 & .30351 & .29853 & .29408 & .28810 & .28142 & .27191 & .25732 & .18922 & \sp & \sp \\
67 & .63212 & .55374 & .49658 & .46956 & .44146 & .42506 & .40848 & .39728 & .38516 & .37687 & .36773 & .36119 & .35395 & .34856 & .34262 & .33815 & .33280 & .32865 & .32423 & .32093 & .31762 & .31634 & .31597 & .31230 & .30839 & .30414 & .29964 & .29508 & .29002 & .28386 & .27616 & .26513 & .24368 & \sp & \sp \\
68 & .63212 & .55374 & .49658 & .46956 & .44146 & .42506 & .40848 & .39728 & .38516 & .37687 & .36773 & .36119 & .35395 & .34856 & .34261 & .33815 & .33285 & .32868 & .32419 & .32085 & .31734 & .31547 & .31564 & .31286 & .30857 & .30496 & .30015 & .29613 & .29109 & .28582 & .27868 & .26983 & .25497 & .18792 & \sp \\
69 & .63212 & .55374 & .49658 & .46956 & .44146 & .42506 & .40848 & .39728 & .38516 & .37687 & .36773 & .36119 & .35396 & .34855 & .34260 & .33814 & .33301 & .32877 & .32426 & .32084 & .31725 & .31497 & .31634 & .31273 & .30919 & .30535 & .30118 & .29685 & .29244 & .28743 & .28138 & .27379 & .26293 & .24172 & \sp \\
70 & .63212 & .55374 & .49658 & .46956 & .44146 & .42506 & .40848 & .39728 & .38516 & .37687 & .36773 & .36119 & .35396 & .34855 & .34259 & .33814 & .33303 & .32881 & .32426 & .32080 & .31710 & .31463 & .31289 & .31318 & .30921 & .30602 & .30161 & .29769 & .29313 & .28886 & .28304 & .27657 & .26734 & .25312 & .18649 \\
\midrule
$\infty$ &
.63212 &  .55374 &  .49658 &  .46956 &  .44146 &  .42506 &  .40848 &  
.39728 &  .38516 &  .37687 &  .36773 &  .36119 &  .35396 &  .34855 &  
.34256 &  .33807 &  .33297 &  .32908 &  .32473 &  .32132 &  .31750 &  
.31450 &  .31110 &  .30842 &  .30539 &  .30295 &  .30022 &  .29801 &  
.29551 &  .29349 \\
\bottomrule
\end{tabular}}
\end{center}
\end{landscape}

\newpage

\newgeometry{margin=6pt}
\begin{landscape}
\section*{Appendix D: $p(n,k) = 1-i(n,k)$ for $n \le 70$ and $p(\infty,k) = 1-i(\infty,k)$
(values are rounded to 5 decimal places)}

\begin{center}
\scalebox{0.52}{
\renewcommand{\sp}{.}
\begin{tabular}{llllllllllllllllllllllllllllllllllll}\toprule 
$n \backslash k$ & 1 & 2 & 3 & 4 & 5 & 6 & 7 & 8 & 9 & 10 & 11 & 12 & 13 & 14 & 15 & 16 & 17 & 18 & 19 & 20 & 21 & 22 & 23 & 24 & 25 & 26 & 27 & 28 & 29 & 30 & 31 & 32 & 33 & 34 & 35 \\ \midrule
2 & .50000 & \sp & \sp & \sp & \sp & \sp & \sp & \sp & \sp & \sp & \sp & \sp & \sp & \sp & \sp & \sp & \sp & \sp & \sp & \sp & \sp & \sp & \sp & \sp & \sp & \sp & \sp & \sp & \sp & \sp & \sp & \sp & \sp & \sp & \sp \\
3 & .33333 & \sp & \sp & \sp & \sp & \sp & \sp & \sp & \sp & \sp & \sp & \sp & \sp & \sp & \sp & \sp & \sp & \sp & \sp & \sp & \sp & \sp & \sp & \sp & \sp & \sp & \sp & \sp & \sp & \sp & \sp & \sp & \sp & \sp & \sp \\
4 & .37500 & .58333 & \sp & \sp & \sp & \sp & \sp & \sp & \sp & \sp & \sp & \sp & \sp & \sp & \sp & \sp & \sp & \sp & \sp & \sp & \sp & \sp & \sp & \sp & \sp & \sp & \sp & \sp & \sp & \sp & \sp & \sp & \sp & \sp & \sp \\
5 & .36667 & .45000 & \sp & \sp & \sp & \sp & \sp & \sp & \sp & \sp & \sp & \sp & \sp & \sp & \sp & \sp & \sp & \sp & \sp & \sp & \sp & \sp & \sp & \sp & \sp & \sp & \sp & \sp & \sp & \sp & \sp & \sp & \sp & \sp & \sp \\
6 & .36806 & .42222 & .63750 & \sp & \sp & \sp & \sp & \sp & \sp & \sp & \sp & \sp & \sp & \sp & \sp & \sp & \sp & \sp & \sp & \sp & \sp & \sp & \sp & \sp & \sp & \sp & \sp & \sp & \sp & \sp & \sp & \sp & \sp & \sp & \sp \\
7 & .36786 & .44841 & .50952 & \sp & \sp & \sp & \sp & \sp & \sp & \sp & \sp & \sp & \sp & \sp & \sp & \sp & \sp & \sp & \sp & \sp & \sp & \sp & \sp & \sp & \sp & \sp & \sp & \sp & \sp & \sp & \sp & \sp & \sp & \sp & \sp \\
8 & .36788 & .44911 & .49963 & .66230 & \sp & \sp & \sp & \sp & \sp & \sp & \sp & \sp & \sp & \sp & \sp & \sp & \sp & \sp & \sp & \sp & \sp & \sp & \sp & \sp & \sp & \sp & \sp & \sp & \sp & \sp & \sp & \sp & \sp & \sp & \sp \\
9 & .36788 & .44576 & .48522 & .55181 & \sp & \sp & \sp & \sp & \sp & \sp & \sp & \sp & \sp & \sp & \sp & \sp & \sp & \sp & \sp & \sp & \sp & \sp & \sp & \sp & \sp & \sp & \sp & \sp & \sp & \sp & \sp & \sp & \sp & \sp & \sp \\
10 & .36788 & .44601 & .50533 & .53230 & .68679 & \sp & \sp & \sp & \sp & \sp & \sp & \sp & \sp & \sp & \sp & \sp & \sp & \sp & \sp & \sp & \sp & \sp & \sp & \sp & \sp & \sp & \sp & \sp & \sp & \sp & \sp & \sp & \sp & \sp & \sp \\
11 & .36788 & .44633 & .50476 & .52233 & .58029 & \sp & \sp & \sp & \sp & \sp & \sp & \sp & \sp & \sp & \sp & \sp & \sp & \sp & \sp & \sp & \sp & \sp & \sp & \sp & \sp & \sp & \sp & \sp & \sp & \sp & \sp & \sp & \sp & \sp & \sp \\
12 & .36788 & .44628 & .50495 & .51473 & .56550 & .70123 & \sp & \sp & \sp & \sp & \sp & \sp & \sp & \sp & \sp & \sp & \sp & \sp & \sp & \sp & \sp & \sp & \sp & \sp & \sp & \sp & \sp & \sp & \sp & \sp & \sp & \sp & \sp & \sp & \sp \\
13 & .36788 & .44625 & .50300 & .53009 & .55395 & .60280 & \sp & \sp & \sp & \sp & \sp & \sp & \sp & \sp & \sp & \sp & \sp & \sp & \sp & \sp & \sp & \sp & \sp & \sp & \sp & \sp & \sp & \sp & \sp & \sp & \sp & \sp & \sp & \sp & \sp \\
14 & .36788 & .44626 & .50327 & .53204 & .55095 & .58355 & .71497 & \sp & \sp & \sp & \sp & \sp & \sp & \sp & \sp & \sp & \sp & \sp & \sp & \sp & \sp & \sp & \sp & \sp & \sp & \sp & \sp & \sp & \sp & \sp & \sp & \sp & \sp & \sp & \sp \\
15 & .36788 & .44626 & .50333 & .53179 & .54534 & .57687 & .62073 & \sp & \sp & \sp & \sp & \sp & \sp & \sp & \sp & \sp & \sp & \sp & \sp & \sp & \sp & \sp & \sp & \sp & \sp & \sp & \sp & \sp & \sp & \sp & \sp & \sp & \sp & \sp & \sp \\
16 & .36788 & .44626 & .50346 & .53151 & .55850 & .56946 & .60501 & .72448 & \sp & \sp & \sp & \sp & \sp & \sp & \sp & \sp & \sp & \sp & \sp & \sp & \sp & \sp & \sp & \sp & \sp & \sp & \sp & \sp & \sp & \sp & \sp & \sp & \sp & \sp & \sp \\
17 & .36788 & .44626 & .50343 & .53024 & .55925 & .56699 & .59374 & .63570 & \sp & \sp & \sp & \sp & \sp & \sp & \sp & \sp & \sp & \sp & \sp & \sp & \sp & \sp & \sp & \sp & \sp & \sp & \sp & \sp & \sp & \sp & \sp & \sp & \sp & \sp & \sp \\
18 & .36788 & .44626 & .50342 & .53021 & .55999 & .56309 & .59101 & .61801 & .73432 & \sp & \sp & \sp & \sp & \sp & \sp & \sp & \sp & \sp & \sp & \sp & \sp & \sp & \sp & \sp & \sp & \sp & \sp & \sp & \sp & \sp & \sp & \sp & \sp & \sp & \sp \\
19 & .36788 & .44626 & .50341 & .53034 & .55944 & .57406 & .58482 & .60903 & .64801 & \sp & \sp & \sp & \sp & \sp & \sp & \sp & \sp & \sp & \sp & \sp & \sp & \sp & \sp & \sp & \sp & \sp & \sp & \sp & \sp & \sp & \sp & \sp & \sp & \sp & \sp \\
20 & .36788 & .44626 & .50342 & .53040 & .55927 & .57575 & .58414 & .60347 & .63266 & .74119 & \sp & \sp & \sp & \sp & \sp & \sp & \sp & \sp & \sp & \sp & \sp & \sp & \sp & \sp & \sp & \sp & \sp & \sp & \sp & \sp & \sp & \sp & \sp & \sp & \sp \\
21 & .36788 & .44626 & .50342 & .53047 & .55833 & .57584 & .58071 & .59974 & .62154 & .65898 & \sp & \sp & \sp & \sp & \sp & \sp & \sp & \sp & \sp & \sp & \sp & \sp & \sp & \sp & \sp & \sp & \sp & \sp & \sp & \sp & \sp & \sp & \sp & \sp & \sp \\
22 & .36788 & .44626 & .50342 & .53046 & .55840 & .57609 & .59062 & .59609 & .61809 & .64221 & .74847 & \sp & \sp & \sp & \sp & \sp & \sp & \sp & \sp & \sp & \sp & \sp & \sp & \sp & \sp & \sp & \sp & \sp & \sp & \sp & \sp & \sp & \sp & \sp & \sp \\
23 & .36788 & .44626 & .50342 & .53045 & .55841 & .57571 & .59175 & .59528 & .61292 & .63269 & .66823 & \sp & \sp & \sp & \sp & \sp & \sp & \sp & \sp & \sp & \sp & \sp & \sp & \sp & \sp & \sp & \sp & \sp & \sp & \sp & \sp & \sp & \sp & \sp & \sp \\
24 & .36788 & .44626 & .50342 & .53044 & .55850 & .57550 & .59251 & .59274 & .61131 & .62738 & .65323 & .75382 & \sp & \sp & \sp & \sp & \sp & \sp & \sp & \sp & \sp & \sp & \sp & \sp & \sp & \sp & \sp & \sp & \sp & \sp & \sp & \sp & \sp & \sp & \sp \\
25 & .36788 & .44626 & .50342 & .53044 & .55853 & .57481 & .59234 & .60138 & .60810 & .62420 & .64253 & .67658 & \sp & \sp & \sp & \sp & \sp & \sp & \sp & \sp & \sp & \sp & \sp & \sp & \sp & \sp & \sp & \sp & \sp & \sp & \sp & \sp & \sp & \sp & \sp \\
26 & .36788 & .44626 & .50342 & .53044 & .55857 & .57477 & .59253 & .60288 & .60801 & .62030 & .63829 & .66065 & .75948 & \sp & \sp & \sp & \sp & \sp & \sp & \sp & \sp & \sp & \sp & \sp & \sp & \sp & \sp & \sp & \sp & \sp & \sp & \sp & \sp & \sp & \sp \\
27 & .36788 & .44626 & .50342 & .53044 & .55855 & .57483 & .59212 & .60337 & .60566 & .61883 & .63364 & .65139 & .68398 & \sp & \sp & \sp & \sp & \sp & \sp & \sp & \sp & \sp & \sp & \sp & \sp & \sp & \sp & \sp & \sp & \sp & \sp & \sp & \sp & \sp & \sp \\
28 & .36788 & .44626 & .50342 & .53044 & .55855 & .57486 & .59198 & .60367 & .61356 & .61654 & .63197 & .64542 & .66939 & .76383 & \sp & \sp & \sp & \sp & \sp & \sp & \sp & \sp & \sp & \sp & \sp & \sp & \sp & \sp & \sp & \sp & \sp & \sp & \sp & \sp & \sp \\
29 & .36788 & .44626 & .50342 & .53044 & .55854 & .57491 & .59139 & .60354 & .61470 & .61635 & .62851 & .64220 & .65908 & .69069 & \sp & \sp & \sp & \sp & \sp & \sp & \sp & \sp & \sp & \sp & \sp & \sp & \sp & \sp & \sp & \sp & \sp & \sp & \sp & \sp & \sp \\
30 & .36788 & .44626 & .50342 & .53044 & .55854 & .57493 & .59139 & .60356 & .61548 & .61452 & .62790 & .63948 & .65456 & .67548 & .76847 & \sp & \sp & \sp & \sp & \sp & \sp & \sp & \sp & \sp & \sp & \sp & \sp & \sp & \sp & \sp & \sp & \sp & \sp & \sp & \sp \\
31 & .36788 & .44626 & .50342 & .53044 & .55854 & .57496 & .59139 & .60325 & .61552 & .62155 & .62578 & .63721 & .64978 & .66630 & .69670 & \sp & \sp & \sp & \sp & \sp & \sp & \sp & \sp & \sp & \sp & \sp & \sp & \sp & \sp & \sp & \sp & \sp & \sp & \sp & \sp \\
32 & .36788 & .44626 & .50342 & .53044 & .55854 & .57495 & .59144 & .60309 & .61574 & .62298 & .62591 & .63478 & .64798 & .66064 & .68259 & .77205 & \sp & \sp & \sp & \sp & \sp & \sp & \sp & \sp & \sp & \sp & \sp & \sp & \sp & \sp & \sp & \sp & \sp & \sp & \sp \\
33 & .36788 & .44626 & .50342 & .53044 & .55854 & .57495 & .59146 & .60263 & .61553 & .62350 & .62416 & .63408 & .64518 & .65682 & .67275 & .70230 & \sp & \sp & \sp & \sp & \sp & \sp & \sp & \sp & \sp & \sp & \sp & \sp & \sp & \sp & \sp & \sp & \sp & \sp & \sp \\
34 & .36788 & .44626 & .50342 & .53044 & .55854 & .57494 & .59150 & .60259 & .61555 & .62382 & .63083 & .63251 & .64387 & .65401 & .66798 & .68769 & .77596 & \sp & \sp & \sp & \sp & \sp & \sp & \sp & \sp & \sp & \sp & \sp & \sp & \sp & \sp & \sp & \sp & \sp & \sp \\
35 & .36788 & .44626 & .50342 & .53044 & .55854 & .57494 & .59152 & .60261 & .61525 & .62386 & .63187 & .63255 & .64163 & .65219 & .66310 & .67861 & .70733 & \sp & \sp & \sp & \sp & \sp & \sp & \sp & \sp & \sp & \sp & \sp & \sp & \sp & \sp & \sp & \sp & \sp & \sp \\
36 & .36788 & .44626 & .50342 & .53044 & .55854 & .57494 & .59154 & .60263 & .61515 & .62394 & .63264 & .63110 & .64139 & .64991 & .66121 & .67312 & .69366 & .77904 & \sp & \sp & \sp & \sp & \sp & \sp & \sp & \sp & \sp & \sp & \sp & \sp & \sp & \sp & \sp & \sp & \sp \\
37 & .36788 & .44626 & .50342 & .53044 & .55854 & .57494 & .59153 & .60267 & .61475 & .62377 & .63277 & .63714 & .63984 & .64863 & .65844 & .66909 & .68406 & .71206 & \sp & \sp & \sp & \sp & \sp & \sp & \sp & \sp & \sp & \sp & \sp & \sp & \sp & \sp & \sp & \sp & \sp \\
38 & .36788 & .44626 & .50342 & .53044 & .55854 & .57494 & .59153 & .60268 & .61473 & .62374 & .63301 & .63841 & .64012 & .64691 & .65733 & .66617 & .67933 & .69799 & .78235 & \sp & \sp & \sp & \sp & \sp & \sp & \sp & \sp & \sp & \sp & \sp & \sp & \sp & \sp & \sp & \sp \\
39 & .36788 & .44626 & .50342 & .53044 & .55854 & .57494 & .59152 & .60271 & .61472 & .62351 & .63294 & .63895 & .63870 & .64663 & .65509 & .66430 & .67450 & .68911 & .71639 & \sp & \sp & \sp & \sp & \sp & \sp & \sp & \sp & \sp & \sp & \sp & \sp & \sp & \sp & \sp & \sp \\
40 & .36788 & .44626 & .50342 & .53044 & .55854 & .57494 & .59152 & .60272 & .61475 & .62339 & .63302 & .63932 & .64442 & .64547 & .65440 & .66238 & .67233 & .68375 & .70309 & .78502 & \sp & \sp & \sp & \sp & \sp & \sp & \sp & \sp & \sp & \sp & \sp & \sp & \sp & \sp & \sp \\
41 & .36788 & .44626 & .50342 & .53044 & .55854 & .57494 & .59152 & .60274 & .61476 & .62306 & .63282 & .63940 & .64546 & .64561 & .65278 & .66088 & .66959 & .67962 & .69373 & .72046 & \sp & \sp & \sp & \sp & \sp & \sp & \sp & \sp & \sp & \sp & \sp & \sp & \sp & \sp & \sp \\
42 & .36788 & .44626 & .50342 & .53044 & .55854 & .57494 & .59152 & .60273 & .61480 & .62302 & .63281 & .63950 & .64619 & .64442 & .65274 & .65928 & .66845 & .67659 & .68907 & .70688 & .78791 & \sp & \sp & \sp & \sp & \sp & \sp & \sp & \sp & \sp & \sp & \sp & \sp & \sp & \sp \\
43 & .36788 & .44626 & .50342 & .53044 & .55854 & .57494 & .59152 & .60273 & .61481 & .62303 & .63259 & .63946 & .64637 & .64971 & .65157 & .65850 & .66646 & .67461 & .68431 & .69814 & .72422 & \sp & \sp & \sp & \sp & \sp & \sp & \sp & \sp & \sp & \sp & \sp & \sp & \sp & \sp \\
44 & .36788 & .44626 & .50342 & .53044 & .55854 & .57494 & .59152 & .60273 & .61483 & .62304 & .63250 & .63948 & .64664 & .65085 & .65188 & .65726 & .66550 & .67277 & .68190 & .69289 & .71129 & .79024 & \sp & \sp & \sp & \sp & \sp & \sp & \sp & \sp & \sp & \sp & \sp & \sp & \sp \\
45 & .36788 & .44626 & .50342 & .53044 & .55854 & .57494 & .59152 & .60272 & .61484 & .62306 & .63220 & .63932 & .64662 & .65141 & .65071 & .65715 & .66396 & .67155 & .67921 & .68878 & .70220 & .72780 & \sp & \sp & \sp & \sp & \sp & \sp & \sp & \sp & \sp & \sp & \sp & \sp & \sp \\
46 & .36788 & .44626 & .50342 & .53044 & .55854 & .57494 & .59152 & .60272 & .61485 & .62307 & .63218 & .63928 & .64673 & .65180 & .65575 & .65622 & .66355 & .66984 & .67797 & .68574 & .69752 & .71461 & .79280 & \sp & \sp & \sp & \sp & \sp & \sp & \sp & \sp & \sp & \sp & \sp & \sp \\
47 & .36788 & .44626 & .50342 & .53044 & .55854 & .57494 & .59152 & .60272 & .61485 & .62310 & .63217 & .63910 & .64664 & .65192 & .65671 & .65645 & .66230 & .66889 & .67606 & .68356 & .69288 & .70603 & .73110 & \sp & \sp & \sp & \sp & \sp & \sp & \sp & \sp & \sp & \sp & \sp & \sp \\
48 & .36788 & .44626 & .50342 & .53044 & .55854 & .57494 & .59152 & .60272 & .61484 & .62311 & .63219 & .63901 & .64665 & .65205 & .65739 & .65544 & .66238 & .66771 & .67526 & .68169 & .69042 & .70093 & .71851 & .79487 & \sp & \sp & \sp & \sp & \sp & \sp & \sp & \sp & \sp & \sp & \sp \\
49 & .36788 & .44626 & .50342 & .53044 & .55854 & .57494 & .59152 & .60272 & .61484 & .62312 & .63219 & .63875 & .64650 & .65205 & .65764 & .66013 & .66146 & .66723 & .67358 & .68044 & .68757 & .69680 & .70965 & .73427 & \sp & \sp & \sp & \sp & \sp & \sp & \sp & \sp & \sp & \sp & \sp \\
50 & .36788 & .44626 & .50342 & .53044 & .55854 & .57494 & .59152 & .60272 & .61484 & .62313 & .63222 & .63871 & .64647 & .65209 & .65792 & .66119 & .66178 & .66625 & .67300 & .67907 & .68629 & .69380 & .70500 & .72145 & .79714 & \sp & \sp & \sp & \sp & \sp & \sp & \sp & \sp & \sp & \sp \\
51 & .36788 & .44626 & .50342 & .53044 & .55854 & .57494 & .59152 & .60272 & .61484 & .62314 & .63223 & .63872 & .64629 & .65202 & .65793 & .66175 & .66078 & .66625 & .67182 & .67798 & .68444 & .69150 & .70046 & .71307 & .73722 & \sp & \sp & \sp & \sp & \sp & \sp & \sp & \sp & \sp & \sp \\
52 & .36788 & .44626 & .50342 & .53044 & .55854 & .57494 & .59152 & .60272 & .61484 & .62314 & .63225 & .63872 & .64622 & .65201 & .65805 & .66211 & .66529 & .66551 & .67156 & .67670 & .68359 & .68961 & .69795 & .70801 & .72492 & .79901 & \sp & \sp & \sp & \sp & \sp & \sp & \sp & \sp & \sp \\
53 & .36788 & .44626 & .50342 & .53044 & .55854 & .57494 & .59152 & .60272 & .61483 & .62313 & .63225 & .63873 & .64599 & .65187 & .65801 & .66229 & .66619 & .66576 & .67059 & .67605 & .68215 & .68824 & .69512 & .70394 & .71626 & .74005 & \sp & \sp & \sp & \sp & \sp & \sp & \sp & \sp & \sp \\
54 & .36788 & .44626 & .50342 & .53044 & .55854 & .57494 & .59152 & .60272 & .61483 & .62313 & .63227 & .63874 & .64596 & .65183 & .65805 & .66244 & .66685 & .66489 & .67073 & .67515 & .68143 & .68692 & .69368 & .70100 & .71167 & .72758 & .80105 & \sp & \sp & \sp & \sp & \sp & \sp & \sp & \sp \\
55 & .36788 & .44626 & .50342 & .53044 & .55854 & .57494 & .59152 & .60272 & .61483 & .62313 & .63227 & .63876 & .64595 & .65169 & .65796 & .66246 & .66711 & .66911 & .66995 & .67485 & .68016 & .68602 & .69185 & .69861 & .70722 & .71934 & .74270 & \sp & \sp & \sp & \sp & \sp & \sp & \sp & \sp \\
56 & .36788 & .44626 & .50342 & .53044 & .55854 & .57494 & .59152 & .60272 & .61484 & .62313 & .63228 & .63877 & .64596 & .65161 & .65795 & .66253 & .66738 & .67009 & .67030 & .67405 & .67977 & .68475 & .69094 & .69661 & .70467 & .71439 & .73068 & .80273 & \sp & \sp & \sp & \sp & \sp & \sp & \sp \\
57 & .36788 & .44626 & .50342 & .53044 & .55854 & .57494 & .59152 & .60272 & .61484 & .62313 & .63228 & .63879 & .64596 & .65141 & .65782 & .66249 & .66744 & .67063 & .66942 & .67412 & .67884 & .68399 & .68956 & .69526 & .70189 & .71034 & .72224 & .74526 & \sp & \sp & \sp & \sp & \sp & \sp & \sp \\
58 & .36788 & .44626 & .50342 & .53044 & .55854 & .57494 & .59152 & .60272 & .61484 & .62313 & .63228 & .63880 & .64598 & .65137 & .65779 & .66250 & .66758 & .67099 & .67350 & .67350 & .67870 & .68298 & .68894 & .69391 & .70033 & .70744 & .71766 & .73309 & .80459 & \sp & \sp & \sp & \sp & \sp & \sp \\
59 & .36788 & .44626 & .50342 & .53044 & .55854 & .57494 & .59152 & .60272 & .61484 & .62313 & .63227 & .63881 & .64599 & .65137 & .65764 & .66242 & .66755 & .67118 & .67436 & .67377 & .67790 & .68255 & .68764 & .69297 & .69845 & .70503 & .71329 & .72500 & .74767 & \sp & \sp & \sp & \sp & \sp & \sp \\
60 & .36788 & .44626 & .50342 & .53044 & .55854 & .57494 & .59152 & .60272 & .61484 & .62313 & .63227 & .63881 & .64600 & .65137 & .65758 & .66240 & .66761 & .67134 & .67498 & .67300 & .67807 & .68182 & .68715 & .69194 & .69754 & .70305 & .71077 & .72014 & .73591 & .80613 & \sp & \sp & \sp & \sp & \sp \\
61 & .36788 & .44626 & .50342 & .53044 & .55854 & .57494 & .59152 & .60272 & .61484 & .62313 & .63227 & .63882 & .64601 & .65138 & .65739 & .66228 & .66755 & .67139 & .67524 & .67685 & .67742 & .68162 & .68615 & .69109 & .69616 & .70156 & .70798 & .71612 & .72763 & .74999 & \sp & \sp & \sp & \sp & \sp \\
62 & .36788 & .44626 & .50342 & .53044 & .55854 & .57494 & .59152 & .60272 & .61484 & .62313 & .63227 & .63881 & .64603 & .65139 & .65737 & .66224 & .66756 & .67146 & .67552 & .67776 & .67776 & .68096 & .68588 & .69008 & .69551 & .70020 & .70638 & .71326 & .72311 & .73810 & .80781 & \sp & \sp & \sp & \sp \\
63 & .36788 & .44626 & .50342 & .53044 & .55854 & .57494 & .59152 & .60272 & .61484 & .62313 & .63227 & .63881 & .64603 & .65140 & .65735 & .66212 & .66748 & .67145 & .67560 & .67828 & .67698 & .68107 & .68513 & .68953 & .69439 & .69923 & .70445 & .71085 & .71881 & .73016 & .75219 & \sp & \sp & \sp & \sp \\
64 & .36788 & .44626 & .50342 & .53044 & .55854 & .57494 & .59152 & .60272 & .61484 & .62313 & .63227 & .63881 & .64604 & .65141 & .65736 & .66206 & .66746 & .67148 & .67573 & .67864 & .68072 & .68054 & .68507 & .68873 & .69379 & .69822 & .70348 & .70886 & .71629 & .72538 & .74066 & .80923 & \sp & \sp & \sp \\
65 & .36788 & .44626 & .50342 & .53044 & .55854 & .57494 & .59152 & .60272 & .61484 & .62313 & .63227 & .63881 & .64604 & .65142 & .65736 & .66189 & .66735 & .67143 & .67574 & .67883 & .68152 & .68082 & .68439 & .68842 & .69278 & .69749 & .70212 & .70729 & .71354 & .72138 & .73255 & .75431 & \sp & \sp & \sp \\
66 & .36788 & .44626 & .50342 & .53044 & .55854 & .57494 & .59152 & .60272 & .61484 & .62313 & .63227 & .63881 & .64605 & .65143 & .65738 & .66186 & .66732 & .67142 & .67582 & .67900 & .68211 & .68013 & .68458 & .68782 & .69243 & .69649 & .70147 & .70592 & .71190 & .71858 & .72809 & .74268 & .81078 & \sp & \sp \\
67 & .36788 & .44626 & .50342 & .53044 & .55854 & .57494 & .59152 & .60272 & .61484 & .62313 & .63227 & .63881 & .64605 & .65144 & .65738 & .66185 & .66720 & .67135 & .67577 & .67907 & .68238 & .68366 & .68403 & .68770 & .69161 & .69586 & .70036 & .70492 & .70998 & .71614 & .72384 & .73487 & .75632 & \sp & \sp \\
68 & .36788 & .44626 & .50342 & .53044 & .55854 & .57494 & .59152 & .60272 & .61484 & .62313 & .63227 & .63881 & .64605 & .65144 & .65739 & .66185 & .66715 & .67132 & .67581 & .67915 & .68266 & .68453 & .68436 & .68714 & .69143 & .69504 & .69985 & .70387 & .70891 & .71418 & .72132 & .73017 & .74503 & .81208 & \sp \\
69 & .36788 & .44626 & .50342 & .53044 & .55854 & .57494 & .59152 & .60272 & .61484 & .62313 & .63227 & .63881 & .64604 & .65145 & .65740 & .66186 & .66699 & .67123 & .67574 & .67916 & .68275 & .68503 & .68366 & .68727 & .69081 & .69465 & .69882 & .70315 & .70756 & .71257 & .71862 & .72621 & .73707 & .75828 & \sp \\
70 & .36788 & .44626 & .50342 & .53044 & .55854 & .57494 & .59152 & .60272 & .61484 & .62313 & .63227 & .63881 & .64604 & .65145 & .65741 & .66186 & .66697 & .67119 & .67574 & .67920 & .68290 & .68537 & .68711 & .68682 & .69079 & .69398 & .69839 & .70231 & .70687 & .71114 & .71696 & .72343 & .73266 & .74688 & .81351 \\
\midrule
$\infty$ 
& .36788 & .44626 & .50342 & .53044 & .55854 & .57494 & .59152 & .60272 & .61484 & .62313 & .63227 & .63881 & .64604 & .65145 & 
.65744 & .66193 & .66703 & .67092 & .67527 & .67868 & .68250 & 
.68550 & .68890 & .69158 & .69461 & .69705 & .69978 & .70199 & 
.70449 & .70651 \\
\bottomrule
\end{tabular}}
\end{center}
\end{landscape}

\newpage
\newgeometry{margin=1.35in}

\section*{Acknowledgements}

The authors thank Sean Eberhard for helpful comments
and Jasdeep Kochhar for helpful comments and corrections.

\def\cprime{$'$} \def\Dbar{\leavevmode\lower.6ex\hbox to 0pt{\hskip-.23ex
  \accent"16\hss}D} \def\cprime{$'$}
\providecommand{\bysame}{\leavevmode\hbox to3em{\hrulefill}\thinspace}
\providecommand{\MR}{\relax\ifhmode\unskip\space\fi MR }
\providecommand{\MRhref}[2]{%
  \href{http://www.ams.org/mathscinet-getitem?mr=#1}{#2}
}
\providecommand{\href}[2]{#2}

\end{document}